\documentclass[12pt]{amsart} \font\emailfont=cmtt10

\usepackage{amsmath,amsthm,amsfonts,amscd,flafter,epsf,epsfig,verbatim,color,graphics,pinlabel,graphicx}
\usepackage{tikz} \usepackage{pgflibraryarrows}
\usepackage{pgflibrarysnakes} \usepackage[all]{xy}

\newtheorem{theorem}{Theorem}[section]
\newtheorem{lemma}[theorem]{Lemma}
\newtheorem{proposition}[theorem]{Proposition}

\theoremstyle{definition}

\theoremstyle{remark} \newtheorem{remark}[theorem]{Remark}

 \newcommand{\Z}{\mathbb{Z}}

\newcommand{\marktop}{
  \begin{tikzpicture}[scale=0.25]
    \draw (0,0) -- (1,1); \draw (0,1) -- (0.4,0.6); \draw (1,0) --
    (0.6,0.4); \draw (0,0.7) -- (0.3,1);
  \end{tikzpicture}
}

\newcommand{\markbot}{
  \begin{tikzpicture}[scale=0.25]
    \draw (0,0) -- (1,1); \draw (0,1) -- (0.4,0.6); \draw (1,0) --
    (0.6,0.4); \draw (0.7,0) -- (1,0.3);
  \end{tikzpicture}
}

\newcommand{\vid}{
  \begin{tikzpicture}[scale=0.25]
    \draw (0,1) .. controls (0.4,0.6) and (0.4,0.4) .. (0,0); \draw
    (1,1) .. controls (0.6,0.6) and (0.6,0.4) .. (1,0);
  \end{tikzpicture}
}

\newcommand{\vtoh}{
  \begin{tikzpicture}[scale=0.25]
    \draw (0,1) .. controls (0.4,0.6) and (0.4,0.4) .. (0,0); \draw
    (1,1) .. controls (0.6,0.6) and (0.6,0.4) .. (1,0); \draw [->]
    (0.32,0.5) -- (0.68,0.5);
  \end{tikzpicture}
}

\newcommand{\vld}{
  \begin{tikzpicture}[scale=0.25]
    \draw (0,1) .. controls (0.4,0.6) and (0.4,0.4) .. (0,0); \draw
    (1,1) .. controls (0.6,0.6) and (0.6,0.4) .. (1,0); \fill
    (0.28,0.5) circle (4pt);
  \end{tikzpicture}
}

\newcommand{\vrd}{
  \begin{tikzpicture}[scale=0.25]
    \draw (0,1) .. controls (0.4,0.6) and (0.4,0.4) .. (0,0); \draw
    (1,1) .. controls (0.6,0.6) and (0.6,0.4) .. (1,0); \fill
    (0.72,0.5) circle (4pt);
  \end{tikzpicture}
}

\newcommand{\hid}{
  \begin{tikzpicture}[scale=0.25]
    \draw (0,1) .. controls (0.4,0.6) and (0.6,0.6) .. (1,1); \draw
    (0,0) .. controls (0.4,0.4) and (0.6,0.4) .. (1,0);
  \end{tikzpicture}
}

\newcommand{\htov}{
  \begin{tikzpicture}[scale=0.25]
    \draw (0,1) .. controls (0.4,0.6) and (0.6,0.6) .. (1,1); \draw
    (0,0) .. controls (0.4,0.4) and (0.6,0.4) .. (1,0); \draw [->]
    (0.5,0.32) -- (0.5,0.68);
  \end{tikzpicture}
}

\newcommand{\htd}{
  \begin{tikzpicture}[scale=0.25]
    \draw (0,1) .. controls (0.4,0.6) and (0.6,0.6) .. (1,1); \draw
    (0,0) .. controls (0.4,0.4) and (0.6,0.4) .. (1,0); \fill
    (0.5,0.72) circle (4pt);
  \end{tikzpicture}
}

\newcommand{\hbd}{
  \begin{tikzpicture}[scale=0.25]
    \draw (0,1) .. controls (0.4,0.6) and (0.6,0.6) .. (1,1); \draw
    (0,0) .. controls (0.4,0.4) and (0.6,0.4) .. (1,0); \fill
    (0.5,0.28) circle (4pt);
  \end{tikzpicture}
} \DeclareMathOperator{\id}{id} 

\title[{A sign assignment in totally twisted Khovanov homology}] {A
  sign assignment in totally twisted Khovanov homology}

\author[Andrew Manion]{Andrew Manion} \address{Department of
  Mathematics, Princeton University, New Jersey 08544 \newline
  \indent{\emailfont{amanion@math.princeton.edu}}} \thanks{The author
  was supported by the Department of Defense (DoD) through the
  National Defense Science and Engineering Graduate Fellowship (NDSEG)
  Program.}

\begin{document}

\begin{abstract}
  We lift the characteristic-2 totally twisted Khovanov homology of
  Roberts and Jaeger to a theory with $\Z$ coefficients. The result is
  a complex computing reduced odd Khovanov homology for knots. This
  complex is equivalent to a spanning-tree complex whose differential
  is explicit modulo a sign ambiguity coming from the need to choose a
  sign assignment in the definition of odd Khovanov homology.
\end{abstract}

\maketitle
\section{Introduction}\label{intro}
In the recent paper \cite{Roberts}, Roberts introduced a ``totally
twisted'' version of $\delta$-graded characteristic-2 Khovanov
homology for links. Jaeger~\cite{Jaeger} then showed that for knots,
the reduced totally twisted Khovanov homology actually coincides with
the ordinary reduced Khovanov homology (tensored with a suitable
coefficient field). We show how to extend the totally twisted
construction over $\Z$, in the context of the odd Khovanov homology of
Ozsv{\'a}th, Rasmussen and Szab{\'o}~\cite{Odd}. The result is a chain
complex whose homology computes the reduced odd Khovanov homology of
knots (again, tensored with a suitable ring). Cancelling some
differentials in the complex leads to an equivalent complex whose
generators are in bijection with spanning trees of the Tait graph. The
coefficient of the differential between two spanning trees is
determined up to a sign; the sign ambiguity comes from an analogous
ambiguity in odd Khovanov homology, where one must choose a sign
assignment on the edges of a cube of resolutions.

\subsection{Acknowledgements.} 
The author would like to thank John Baldwin, Zolt{\'a}n Szab{\'o},
Cotton Seed, and Kevin Wilson for several very helpful discussions.

\section{The construction}\label{construction}

We assume the reader is familiar with odd Khovanov homology, as
described in \cite{Odd}. Here we briefly fix notation. Let $D$ be an
$n$-crossing diagram for a link $L$, with marks $m_i$ assigned to
edges. For the general construction, we allow any assignment of marks;
to obtain the relationship with spanning trees in
Section~\ref{spanningtrees}, it will be imporant that each edge has at
least one mark. Let $R$ be the polynomial ring $\Z [x_i]$, with one
variable for each mark. Since we will be working with odd Khovanov
homology, we also want to choose an orientation for each
crossing. Figure~\ref{crossing} below shows one possible choice; the
other has the arrow reversed.

Each crossing in $D$ has a $0$-resolution and a $1$-resolution. A
complete resolution of $D$ gives rise to a diagram with no crossings,
which consists of $k$ unlinked circles. To a complete resolution
$\rho$, associate the group $V_{\rho} = H^*(S^1 \times \ldots \times
S^1)$, where there are $k$ $S^1$ factors. This group is actually a
ring, and if we label the circles of the resolution $a_1, \ldots,
a_k$, a convenient set of multiplicative generators for $V_{\rho}$ may
be labelled $\{a_1, \ldots, a_k\}$ as well. The chain complex
computing odd Khovanov homology may be written
$(C_*^{odd}(D),d_{odd})$, where $C_*^{odd}(D) = \oplus_{\rho \in
  \{0,1\}^n} V_{\rho}$. We refer to \cite{Odd} for the definition of
$d_{odd}$. Here we only remind the reader that while the orientations
on the crossings determine split and merge maps with well-defined
signs, these naive maps do not automatically fit into a differential
satisfying $d_{odd}^2 = 0$. Rather, one must correct by putting signs
on the edges of the cube of resolutions. We will denote the naive maps
collectively as $d'_{odd}$ and reserve the name $d_{odd}$ for the
sign-corrected differential.

\begin{figure}
  \labellist
  \small \hair 2pt \pinlabel $X$ at 173 43 \pinlabel $Y$ at 309 43
  \endlabellist
  \centering
  \includegraphics[scale=0.3]{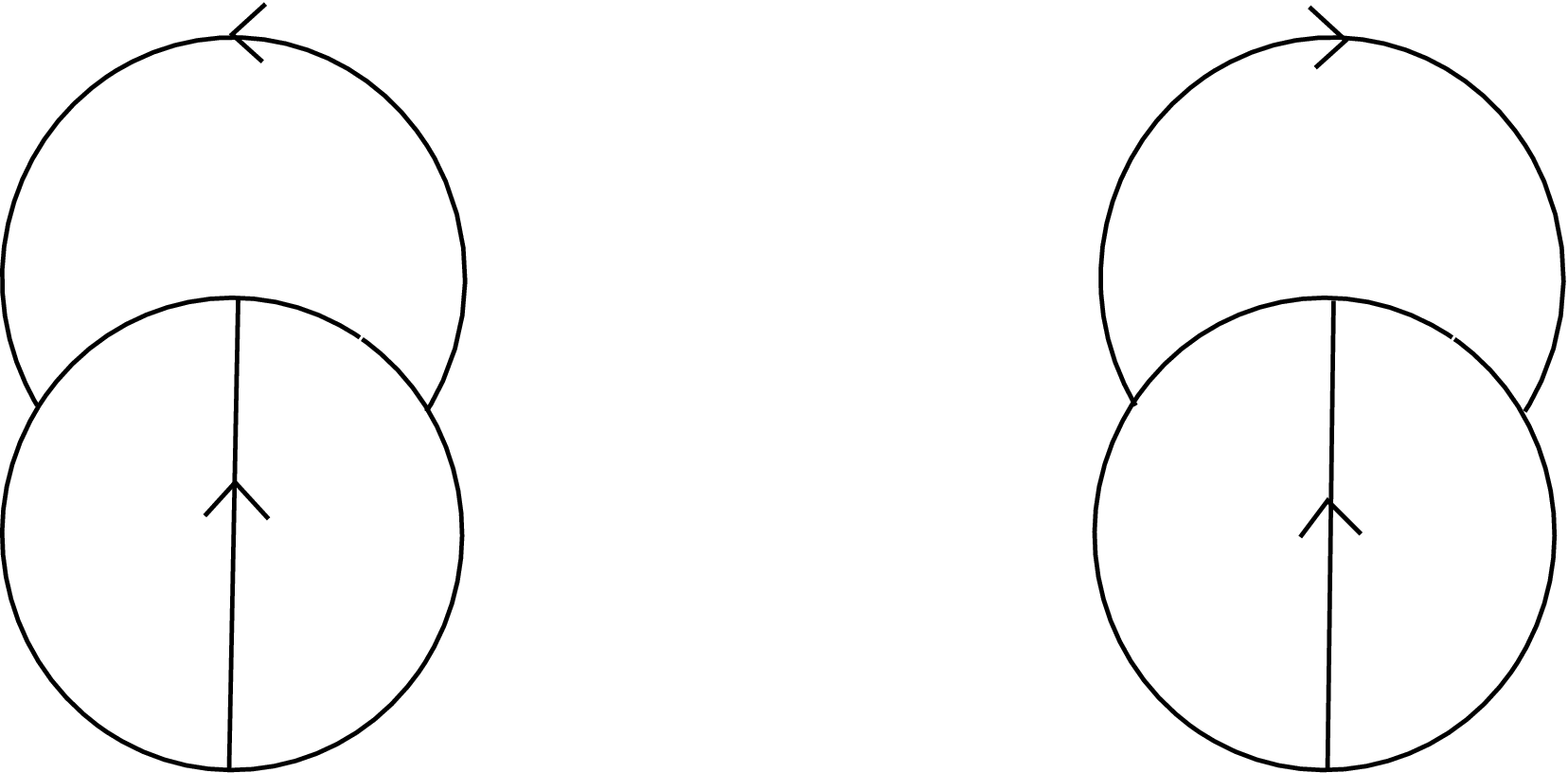}
  \caption{The configurations $X$ and $Y$.}
  \label{XandY}
\end{figure}

The method of correcting the signs makes use of the concept of
2-dimensional (oriented) configurations, which are pictures of a
complete resolution along with two (oriented) arcs such that surgery
along an arc corresponds to switching a crossing. Two examples of
these configurations are depicted in Figure~\ref{XandY}, and many more
are shown below in Figure~\ref{bubbles}. In fact, the configurations
$X$ and $Y$ in Figure~\ref{XandY} are special; they correspond to
faces of the cube of resolutions which both commute and
anticommute. The convention we will use is that that the configuration
labelled $X$ anticommutes and the one labelled $Y$ commutes. We only
mention this here because we will need to deal with 2-dimensional
configurations in Section~\ref{movemarksec}, and there it will be
important that we use this specific convention and not its opposite.

Note that tensoring the odd Khovanov complex with $R$ amounts simply
to taking cohomology with $R$ coefficients when defining $V_{\rho}$,
and we will use these coefficients from here on.

The construction we discuss amounts to defining a differential
$d_{v,\rho}$ on each $V_{\rho}$ (the small $v$ is meant to suggest
``vertical,'' in contrast with the ``horizontal'' maps of
$d_{odd}$). Setting $d_v = \sum_{\rho} d_{v,\rho}$, we consider the
complex $(C_*^{odd},d_v + d_{odd})$. In the remainder of this section,
we will define $d_{v,\rho}$ and show that $(d_v + d_{odd})^2 =
0$. When coefficients are taken modulo $2$, we will get the complex
from \cite{Jaeger}.

Fix a resolution $\rho$ with circles $a_1, \ldots, a_k$. The marks
$m_i$ on $D$ pass to marks on these circles. For each $i$, define $w_i
\in R$ to be the sum of the variables $x_j$ corresponding to those
marks $m_j$ lying on $a_i$. As a preliminary definition, define
\[
d'_{v,\rho} = \sum_i m(w_i a_i, \cdot),
\]
where $m(\cdot, \cdot)$ denotes multiplication in the ring
$V_{\rho}$. (Note that the multiplication is anti-commutative, so
order matters, and we are multiplying by $w_i a_i$ on the left.) It is
clear that $(d'_{v,\rho})^2 = 0$, and the same holds for $d'_v =
\sum_{\rho} d'_{v,\rho}$. Taking coefficients modulo $2$, we get the
twisted complex in Jaeger's form (\cite{Jaeger}): his
dot-multiplication maps correspond to our left multiplication.

The definition of $d'_{v,\rho}$ is only preliminary since we will need
to modify $d'_{v,\rho}$ by an overall sign, depending on $\rho$. We
now describe this modification. It will be based on the following
lemma:

\begin{lemma}\label{naivedv} Write $d'_{odd} = d'_{odd,split} +
  d'_{odd,join}$. Then $d'_v$ anticommutes with $d'_{odd,split}$ and
  commutes with $d'_{odd,join}$.
\end{lemma}

\begin{proof} We will do the split case and leave the (very similar)
  join case to the reader. Consider two resolutions $\rho$ and $\rho'$
  with a nonzero component of $d'_{odd,split}$ between them. Then
  $\rho'$ is obtained from $\rho$ by splitting one circle $a$ into two
  circles $b$ and $c$. Choose the labels $b$ and $c$ such that
  $d'_{odd}(a) = bc$. Denote the passive circles in $\rho$ by
  $\{p_i\}$; then the remaining circles in $\rho'$ may also be
  labelled $\{p_i\}$.

  Homogeneous generators of $V_{\rho}$ take the form $\pi$ or $a\pi$,
  where here $\pi$ denotes any product of the $p_i$. In the following
  computation, all sums over $q$ indicate sums over those passive
  circles $q$ which are not contained in $\pi$.
  \begin{align*}
    d'_v d'_{odd}(\pi) &= d'_v(b-c)\pi = \sum_q w_q q(b-c)\pi,
  \end{align*}
  while
  \begin{align*}
    d'_{odd} d'_v(\pi) &= d'_{odd} (\sum_q w_q q \pi) = \sum_q w_q (b-c)q \pi \\
    &= - d'_v d'_{odd}(\pi).
  \end{align*}
  Similarly,
  \begin{align*}
    d'_v d'_{odd}(a\pi) &= d'_v(bc\pi) = \sum_q w_q qbc\pi
  \end{align*}
  while
  \begin{align*}
    d'_{odd} d'_v(a\pi) &= d'_{odd} (\sum_q w_q qa\pi) = -d'_{odd}(\sum_q w_q aq\pi) = -\sum_q w_q bcq\pi = -\sum_q w_q qbc\pi \\
    &= -d'_v d'_{odd}(a\pi).
  \end{align*}
\end{proof}

The preceding lemma tells us that, to properly define $d_v$, we should
flip the signs on $d'_{v,\rho}$ for some vertices $\rho$ of the cube
of resolutions. For each split edge of the cube, we want the endpoints
to receive the same sign-change, while for each join edge, we want the
endpoints to receive the opposite sign-change. As with odd Khovanov
homology, a cohomological argument allows us to make these choices:

\begin{proposition} It is possible to flip the signs on some of the
  $d'_{v,\rho}$ such that the above conditions hold, and the flips are
  unique up to an overall sign.
\end{proposition}

\begin{figure}
  \begin{center}
    \includegraphics[scale=0.3]{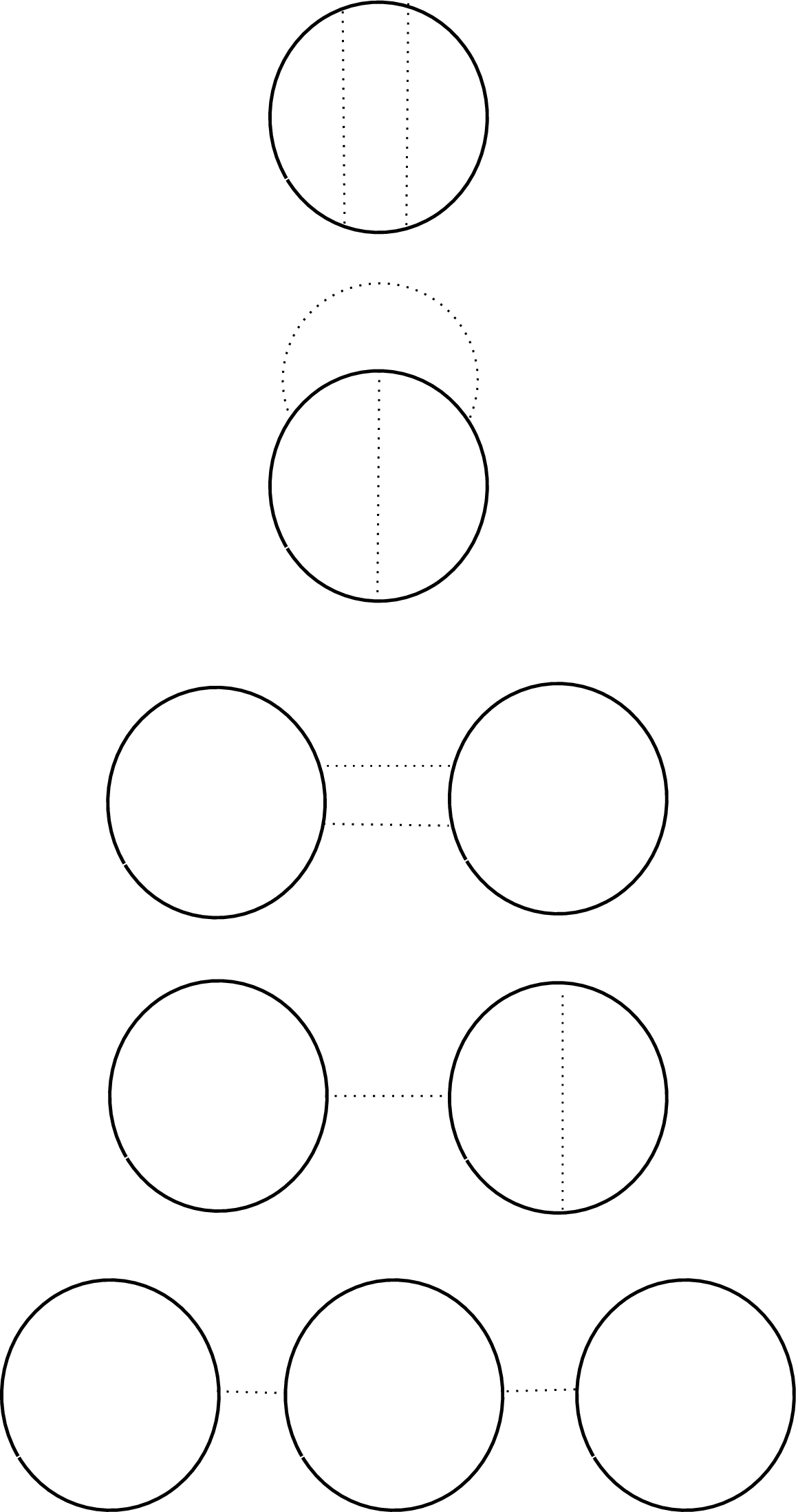}
  \end{center}
  \caption{The five types of unoriented 2-dimensional
    configurations. Dotted lines indicate the arcs of the each
    configuration.}
  \label{unoriented}
\end{figure}

\begin{proof} Consider the usual CW structure of the cube $Q =
  [0,1]^n$ of resolutions, with the $k$-skeleton consisting of the
  $k$-dimensional faces. Define a cochain $\tau \in C^1(Q; \Z/2\Z)$ by
  labelling split edges $0$ and join edges $1$. We want to show that
  $\tau$ is a coboundary, but since $Q$ is contractible, it suffices
  to show $\tau$ is a cocycle. To compute $\delta \tau$, we look at
  the $2$-dimensional faces of $Q$. There are five types of these,
  corresponding to the five possible unoriented 2-dimensional
  configurations. These are shown in Figure~\ref{unoriented}. For each
  of these five configurations, the sum of $\tau$ along the boundary
  edges of the corresponding $2$-dimensional face is $0$ (mod $2$);
  this can easily be checked. Hence $\delta \tau = 0$, so $\tau =
  \delta \sigma$ for some $\sigma \in C^0(Q; \Z/2\Z)$. Flip the sign
  on a vertex $\rho$ if $\sigma(\rho) = 1$, and leave it alone if
  $\sigma(\rho) = 0$. Note that $\sigma$ is unique up to an overall
  sign since $H^0(Q; \Z/2\Z) = \Z/2\Z$.
\end{proof}

If we define $d_v$ by making the appropriate sign flips, then $d_v$
anticommutes with $d'_{odd}$, and of course $d_v^2 = 0$. Finally,
correct the signs in $d'_{odd}$ as in standard odd Khovanov homology
(using another cohomological argument; see \cite{Odd}). It is still
true that $d_v$ anticommutes with $d_{odd}$, so we have $(d_v +
d_{odd})^2 = 0$. This completes the construction of a complex which
reduces to Jaeger's twisted complex modulo 2.

\begin{remark} If we have a basepoint on our link, we can define
  reduced versions of everything above in the standard way.
\end{remark}

\subsection{Invariance.}
The homotopy invariance of $(C_*^{odd},d_v + d_{odd})$ under the
Reidemeister moves can be proved using an argument of Baldwin, similar
to that used in \cite{Baldwin}. For R1 and R2, one writes down the
complexes before and after the move, and then cancels differentials in
the ``before'' complex to obtain the ``after'' complex. Once
invariance under R2 is proven, invariance under R3 amounts to
considering the braid word $xyxy^{-1}x^{-1}y^{-1}$, where $x$ and $y$
are elementary $3$-strand braid group generators, and showing its
appearance in a Khovanov complex is equivalent to the identity. The
relevant ``before'' complex comes from a $64$-vertex cube of
resolutions which was dealt with in \cite{Baldwin}. In our case, the
presence of $d_v$ does not make anything harder for R2 or R3, and R1
does not change much either. In this section we will briefly sketch
the proof for invariance under R1.

\begin{figure}
  \centering
  \labellist
  \small \hair 2pt \pinlabel {Before R1} at 75 62 \pinlabel {After R1}
  at 248 62
  \endlabellist
  \includegraphics[scale=0.5]{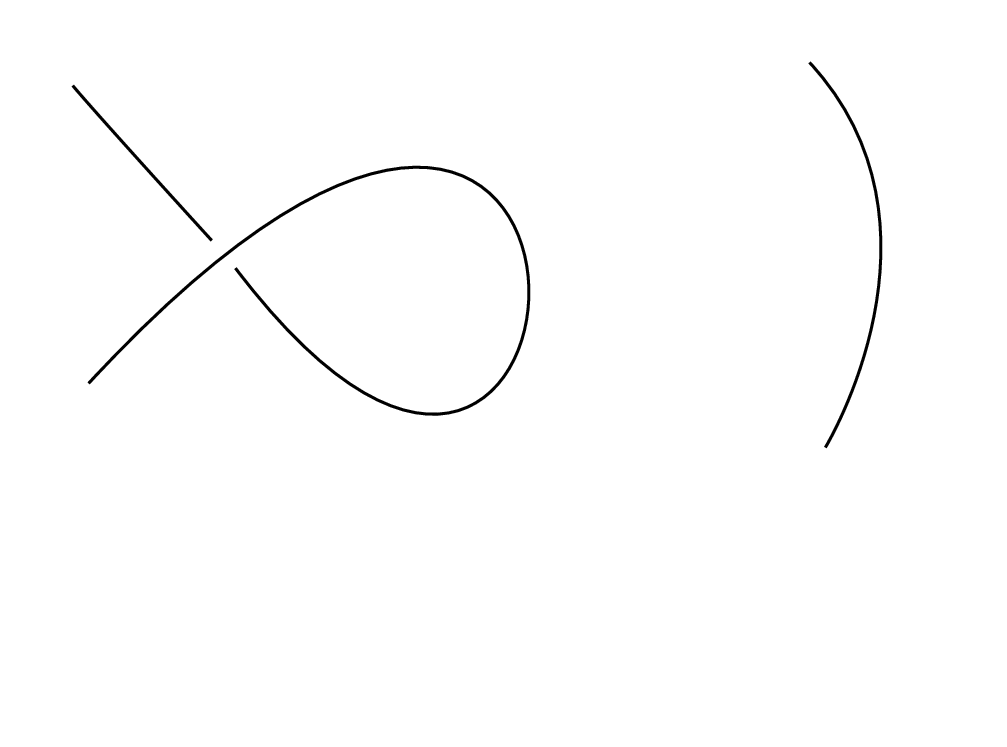}
  \caption{The R1 move.}
  \label{r1}
\end{figure}

Consider an R1 move which undoes a positive kink; see
Figure~\ref{r1}. Write $C_*^{before}$ for the complex before undoing
the kink and $C_*^{after}$ for the complex after performing
R1. Without regard to the differential, $C_*^{before}$ is the sum of
three pieces: $C_*^{before} = C_{0,+} \oplus C_{0,-} \oplus C_1$,
where the $0$ or $1$ indicates the resolution at the crossing in
question and the $+$ (resp. $-$) denotes those generators represented
by monomials not containing (resp. containing) the isolated small
circle in the $0$-resolution. In each of the three local pictures in
question, label the non-closed component as $a$ and the isolated small
circle (in the $0$-pictures) as $b$. Denote the differential from
$C_{0,\pm}$ to $C_1$ by $d_\pm$.

In fact, $d_+$ maps generators of $C_{0,+}$ bijectively to generators
of $C_1$. We would like to cancel all components of $d_+$, leaving
ourselves with $C_{0,-}$ and an induced differential on this
summand. There are some obvious components of this differential,
namely those coming from components internal to $C_{0,-}$ in the whole
complex $C_*^{before}$. These components correspond to almost all of
the differential on $C_*^{after}$ under the bijection sending a
generator $p$ of $C_*^{after}$ to $bp$ in $C_{0,-}$. The only things
missing are the vertical differentials from marks on $b$. We want to
show that the induced differential from cancellation precisely adds in
these missing components.

Indeed, any induced components come from compositions
\[
\xymatrix{C_{0,-} \ar[r]^{d_-} & C_1 \ar[r]^{d_+^{-1}} & C_{0,+}
  \ar[r]^{d_{v,b}} & C_{0,-}}
\]
where the final map comes from the component of $d_v$ associated to
marks on $b$. The map $d_-$ (a join) is nonzero only on elements of
the form $bp$, where $p$ does not contain $a$. We have $d_-(bp) = \pm
ap$, but applying $d_+^{-1}$, the sign cancels and we get $ap \in
C_{0,+}$. Applying $d_{v,b}$ gives us $w \cdot bap$, where $w$ is the
sum of the weights of marks on $b$. This was precisely the component
we were looking for, and invariance under R1 follows.

\section{Relation with odd Khovanov homology.}\label{movemarksec}
Jaeger shows that for knots, his (reduced) complex actually computes
reduced Khovanov homology. We would like to do the same with reduced
odd Khovanov homology. Suppose $L$ is actually a knot $K$ with
basepoint $p$. Following \cite{Jaeger}, the main point is that we can
move marks past crossings without changing the isomorphism type of the
twisted complex. Consider a mark $m$ near a crossing $c$ of $D$, with
local picture \marktop. Let $D'$ be the marked diagram with this local
picture replaced by \markbot. We may assume that $c$ is oriented as in
Figure~\ref{crossing} and that the crossing orientations on $D'$ are
the same as $D$.

\begin{figure}
  \centering
  \includegraphics[scale=0.3]{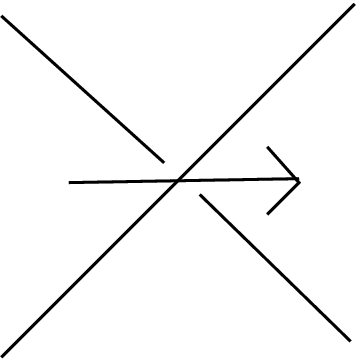}
  \caption{Orientation at the crossing $c$.}
  \label{crossing}
\end{figure}

\begin{theorem}\label{movemark}
  The twisted complexes associated to $D$ and to $D'$ are isomorphic.
\end{theorem}

An analogous statement holds when sliding a mark over a crossing,
rather than under. As in \cite{Jaeger}, Theorem~\ref{movemark} (plus
the analogous statement) immediately implies that the reduced twisted
complex computes reduced odd Khovanov homology for knots. (More
precisely, it computes Khovanov homology tensored with
$\Z[x_i]$). Indeed, one can simply move all the marks to the same edge
as the basepoint, effectively killing $d_v$ and leaving only $d_{odd}$
in the differential. We will now prove Theorem~\ref{movemark}.

\begin{proof}
  In \cite{Jaeger}, with coefficients taken modulo 2, this theorem is
  a purely local computation. Unfortunately, in odd Khovanov homology,
  signs on maps are not determined entirely by local data, so we must
  work a bit harder. Still, Jaeger's chain map (with appropriate
  signs) will work in our situation.

  To define the map, it will be convenient to follow \cite{Jaeger} and
  use local pictures. In this notation, the complex of $D$ will be
  written $\vid \oplus \hid$, where each summand actually represents
  all summands of the total complex of $D$ whose resolution at $c$ is
  as depicted. There is one such summand for each ``outer''
  resolution, i.e. each resolution $\rho$ of all crossings except
  $c$. The complex of $D'$ will also be written as $\vid \oplus \hid$,
  with the same interpretation.

  Fix an outer resolution $\rho$. With respect to the summands \vid
  and \hid, the isomorphism of complexes is given by $F_{\rho}
  := \begin{pmatrix} \vid & \pm \htov \\ & \hid \end{pmatrix}: \vid
  \oplus \hid \to \vid \oplus \hid$. The sign in this formula will
  depend on $\rho$ in a way that will be specified below. Regardless
  of the chosen sign, it is clear that each $F_{\rho}$ is invertible;
  the inverse is the same map with opposite sign. Set $F = \sum_{\rho}
  F_{\rho}$. Then $F$ is invertible, and we only need show that $F$
  commutes with $d_v$ and $d_{odd}$. We may write $F_{\rho} = \id +
  H_{\rho}$, and this will be useful later.

  To specify the sign on $H_{\rho}$, consider
  $(C_*^{odd}(D),d_{odd})$, which is the same as
  $(C_*^{odd}(D'),d_{odd})$ since the only difference between $D$ and
  $D'$ is the placement of a mark. There is a component $d_{c,\rho}$
  of $d_{odd}$ coming from $c$. It has a naive sign from the
  orientation on $c$; write $\sigma(d_{c,\rho}) = 0$ if the actual
  sign agrees with the naive sign and $\sigma(d_{c,\rho}) = 1$
  otherwise. In the process of defining $d_v$ above, we also put signs
  on certain vertices of the complete cube of resolutions. There are
  two such vertices associated to $\rho$; call them $(\rho,0)$ and
  $(\rho,1)$ where the $0$ or $1$ denotes the resolution of $c$. Write
  $\sigma(\rho,i) = 0$ if we did not flip the sign on
  $d_{v,(\rho,i)}$, and write $\sigma(\rho,i) = 1$ if we did. Define
  the sign on $H_{\rho}$ to be
  \begin{equation}\label{hsign}
    \sigma(H_{\rho}) := \sigma(d_{c,\rho}) \cdot \sigma(\rho,1).
  \end{equation}

  To show $F$ is a chain map, we will first consider those components
  of $d_v$ and $d_{odd}$ which correspond to the mark $m$ and the
  crossing in the local picture \marktop (or \markbot). These are the
  components Jaeger deals with in \cite{Jaeger}. The relevant
  commutative diagram in his paper also works in our situation, once
  it is suitably interpreted, and once signs are added.

  \begin{center}
    \scalebox{0.85}{
      $\xymatrix@R=1.5cm@C=3cm{ \vid \oplus \hid \ar[r]^{\begin{pmatrix} w \, \vld \\
            \vtoh & w \, \htd \end{pmatrix}}
        \ar@<0.3ex>[d]_{\begin{pmatrix} \vid & w \htov \\ &
            \hid \end{pmatrix}} & \vid \oplus \hid
        \ar@<0.3ex>[d]^{\begin{pmatrix} \vid & w \htov \\ &
            \hid \end{pmatrix}}
        \\
        \vid \oplus \hid \ar[r]_{\begin{pmatrix} w \, \vrd \\ \vtoh &
            w \, \hbd \end{pmatrix}} & \vid \oplus \hid }$

      $\xymatrix@R=1.5cm@C=3cm{ \vid \oplus \hid \ar[r]^{\begin{pmatrix} w \, \vld \\
            -\vtoh & w \, \htd \end{pmatrix}}
        \ar@<0.3ex>[d]_{\begin{pmatrix} \vid & -w \htov \\ &
            \hid \end{pmatrix}} & \vid \oplus \hid
        \ar@<0.3ex>[d]^{\begin{pmatrix} \vid & - w \htov \\ &
            \hid \end{pmatrix}}
        \\
        \vid \oplus \hid \ar[r]_{\begin{pmatrix} w \, \vrd \\ -\vtoh &
            w \, \hbd \end{pmatrix}} & \vid \oplus \hid }$}
  \end{center}

  \begin{center}
    \scalebox{0.8}{
      $\xymatrix@R=1.5cm@C=3cm{ \vid \oplus \hid \ar[r]^{\begin{pmatrix} w \, \vld \\
            \vtoh & -w \, \htd \end{pmatrix}}
        \ar@<0.3ex>[d]_{\begin{pmatrix} \vid & -w \htov \\ &
            \hid \end{pmatrix}} & \vid \oplus \hid
        \ar@<0.3ex>[d]^{\begin{pmatrix} \vid & -w \htov \\ &
            \hid \end{pmatrix}}
        \\
        \vid \oplus \hid \ar[r]_{\begin{pmatrix} w \, \vrd \\ \vtoh &
            -w \, \hbd \end{pmatrix}} & \vid \oplus \hid }$

      $\xymatrix@R=1.5cm@C=3cm{ \vid \oplus \hid \ar[r]^{\begin{pmatrix} w \, \vld \\
            -\vtoh & -w \, \htd \end{pmatrix}}
        \ar@<0.3ex>[d]_{\begin{pmatrix} \vid & w \htov \\ &
            \hid \end{pmatrix}} & \vid \oplus \hid
        \ar@<0.3ex>[d]^{\begin{pmatrix} \vid & w \htov \\ &
            \hid \end{pmatrix}}
        \\
        \vid \oplus \hid \ar[r]_{\begin{pmatrix} w \, \vrd \\ -\vtoh &
            -w \, \hbd \end{pmatrix}} & \vid \oplus \hid }$ }
  \end{center}

  Above are four copies of the diagram from \cite{Jaeger} with signs
  added. They cover the possible cases when $\sigma(\rho,0) = 0$ (in
  other words, all cases in which the upper-left entries of the
  horizontal maps have a $+$ sign). It turns out that, once we
  consider the case $\sigma(\rho,0) = 0$, the case $\sigma(\rho,0) =
  1$ involves the same set of matrix multiplications, up to an overall
  sign. The diagrams on the left have $\sigma(d_{c,\rho}) = 0$ and the
  ones on the right have $\sigma(d_{c,\rho}) = 1$. In the top
  diagrams, $\sigma(\rho,1) = 0$ (so the crossing change at $c$ is a
  split), and in the bottom ones, $\sigma(\rho,1) = 1$ (so the
  crossing change is a join). On the horizontal maps, a dot represents
  left multiplication by whichever circle contains the dot. The
  vertical maps ($H_{\rho}$) have signs as specified in
  Equation~\eqref{hsign}.

  The reader may check, by multiplying matrices, that the diagrams
  above do commute. A few relations will be needed. First, \vtoh \,
  $\circ$ \htov \, = \htd \, $-$ \hbd, and this holds regardless of
  whether the crossing change is a split or a join. (One sees this by
  explicitly writing down the maps when the first cobordism is a split
  and when it is a join). Analogously, \htov \,$\circ$ \vtoh \, = \vld
  \,$-$ \vrd.

  There are also relations depending on whether the crossing change at
  $c$ is a split (as in the top two diagrams) or a join (as in the
  bottom two diagrams). When the crossing change is a split, we have
  \vrd \,$\circ$ \htov \, = \htov \,$\circ$ \htd, as well as \vld \, =
  \vrd since the dots are on the same circle. When it is a join, we
  have the relations \vrd \, $\circ$ \htov \,$=$ $-$\htov \, $\circ$
  \htd and \hbd \, = \htd. With these relations, one can see that the
  four diagrams above do commute. There are four more diagrams to
  consider, with $\sigma(\rho,0) = 1$, but these computations follow
  from the same set of matrix multiplications. The only difference is
  that some matrices pick up an overall factor of $-1$.

  Next we want to show that $F$ commutes with those components of
  $d_v$ corresponding to marks outside the local crossing picture. We
  may fix an outer resolution $\rho$. Writing $F_{\rho} = \id +
  H_{\rho}$, we can restrict attention to $H_{\rho}$, because $\id$
  commutes with everything outside the local picture. There are
  several cases to consider, and each is an easy algebraic
  computation. We will consider the cases when the crossing change at
  $c$ is a join; the split case is very similar.

  First of all, note that since $\rho$ is fixed, it does not matter
  here which sign was assigned to $H_{\rho}$, so we may assume the
  sign is positive.

  The crossing change at $c$ is a join or a split. Assuming it is a
  join, by our earlier construction, we applied a sign change either
  to $d'_v$ before the crossing change or after (not both); in other
  words, $\sigma(\rho,0) \neq \sigma(\rho,1)$. Hence we want to show
  that $H_{\rho}$ anticommutes with the relevant components of $d'_v$
  (recall that the $'$ indicates the naive signs). The domain of
  $H_{\rho}$ is the resolution $(\rho,1)$. This has one circle, say
  $a$, which intersects the local crossing picture, and possibly
  several other circles (say $\{p_i\}$). $H_{\rho}$ splits $a$ into
  two circles, say $b$ and $c$ (chosen so that $H_{\rho}(a) =
  bc$). Let $\pi$ be a monomial in the $p_i$. Let $m$ be a mark
  outside the local picture; $m$ may lie on any circle $q$ of the
  resolution (possibly $a$). Let $w$ denote the variable associated to
  $m$, and let $d'_{v,q}$ be the component of $d'_v$ coming from $q$.

  If $q$ is not $a$, then
  \[
  d'_{v,q}(H_{\rho}(\pi)) = d'_{v,q}((b-c)\pi) = wq(b-c)\pi,
  \]
  while
  \[
  H_{\rho}(d'_{v,q}(\pi)) = H_{\rho}(wq\pi) = w(b-c)q\pi,
  \]
  which equals $-d'_{v,q}(H_{\rho}(\pi))$ as desired. Similarly,
  \[
  d'_{v,q}(H_{\rho}(a\pi)) = d'_{v,q}(bc\pi) = wqbc\pi,
  \]
  while
  \[
  H_{\rho}(d'_{v,q}(a\pi)) = H_{\rho}(wqa\pi) = -H_{\rho}(waq\pi) =
  -wbcq\pi,
  \]
  which is $-d'_{v,q}(H_{\rho}(a\pi))$ as desired.

  Now, suppose $q = a$ and the mark $m$ lies on $b$ after the
  split. We have
  \[
  d'_{v,q}(H_{\rho}(\pi)) = d'_{v,q}((b-c)\pi) = -wbc\pi,
  \]
  while
  \[
  H_{\rho}(d'_{v,q})(\pi) = H_{\rho}(wa\pi) = wbc\pi,
  \]
  which is $-d'_{v,q}(H_{\rho}(\pi))$. For $a\pi$, the composition
  either way gives zero. If the mark insteads lies on $c$ after the
  split, the computation is exactly analogous, so we have finished
  with the case where the crossing change at $c$ is a join.

  The split case is very similar. We may assume no sign changes were
  applied to $d'_v$ before or after the crossing; since
  $\sigma(\rho,0) = \sigma(\rho,1)$, we may assume both are zero. We
  now want to show $H$ commutes with $d'_v$. We must consider marks on
  the top local circle, on the bottom local circle, and disjoint from
  either circle. After doing the computations in each case, we see
  that $H$ commutes with $d_v$.

  Our final task is to show $F$ (or equivalently $H$) commutes with
  $d_{odd}$, and we need only consider components of $d_{odd}$ coming
  from crossings outside our local picture. Consider an external
  crossing $e$, and fix a resolution of all the other external
  crossings. We get two outer resolutions $\rho$ and $\rho'$, where in
  $\rho$ we take the $0$-resolution at $e$ and in $\rho'$ we take the
  $1$-resolution. Let $\Delta\sigma(H) = \sigma(H_{\rho}) -
  \sigma(H_{\rho'})$ (taken modulo 2). Let $\Delta\sigma(e) =
  \sigma(d_{e,\rho}) - \sigma(d_{e,\rho'})$ with notation similar to
  above. We may define $\Delta\sigma(c)$ analogously.

  Next consider the two-dimensional configuration generated by $c$ and
  $e$ (with all other crossings resolved as we decided above). Call
  this $f$, the ``forward'' configuration. There is a square
  associated to $f$; its sides come from $d'_c$ and $d'_e$ (without
  sign corrections).  Write $a_f = 0$ if it commutes and $a_f = 1$ if
  it anti-commutes; if it does both, use the convention specified in
  Section~\ref{construction}.

  From $f$ we can obtain a ``backward'' configuration by resolving $c$
  and replacing the corresponding oriented arc by one rotated 90
  degrees counter-clockwise. Call this configuration
  $b$. Figure~\ref{bubbles} has many examples of configurations and
  their backwards partners. For the backward configuration, we again
  get a square which either commutes (set $a_b = 0$) or anti-commutes
  (set $a_b = 1$). The sides of this square come from $d'_e$ and $H$
  (taken with the ``naive'' positive sign), since $H$ amounts to doing
  the crossing change at $c$ ``backwards.''

  What we want to show is that $\Delta\sigma(H) + \Delta\sigma(e) +
  a_b = 0$ modulo $2$. The following lemma allows us to do this:
  \begin{lemma}\label{configs}
    For a two-dimensional configuration associated to oriented
    crossings $c$ and $e$ in a diagram, define $f$, $b$, $a_f$, and
    $a_b$ as above ($a_f$ and $a_b$ have values modulo 2).
    \begin{enumerate}
    \item If the arc associated to $e$ is a split in the backward
      configuration $b$, then $a_f = a_b + 1$.
    \item If it is a join, then $a_f = a_b$.
    \end{enumerate}
  \end{lemma}
  \begin{proof}
    There are several cases to be considered; in fact, a diagram is
    more useful than words here. Figure~\ref{bubbles} depicts the
    relevant 2-dimensional configurations. The left column shows the
    forward and backward configurations such that $e$ is a split in
    the backwards configuration. The right column does the same for
    configurations where $e$ is a join in the backwards
    configuration. All configurations are labeled with ``comm.'' if
    the corresponding square commutes and ``anti.''  if it does
    not. The content of the lemma is that these labels are correct
    (each is a simple verification), plus the fact that in the left
    column, a configuration and its backwards partner have opposite
    labels while in the right column they have the same labels. Note
    that our choice of convention for the configurations $X$ and $Y$
    of Figure~\ref{XandY} is needed for the lemma to hold.
  \end{proof}

  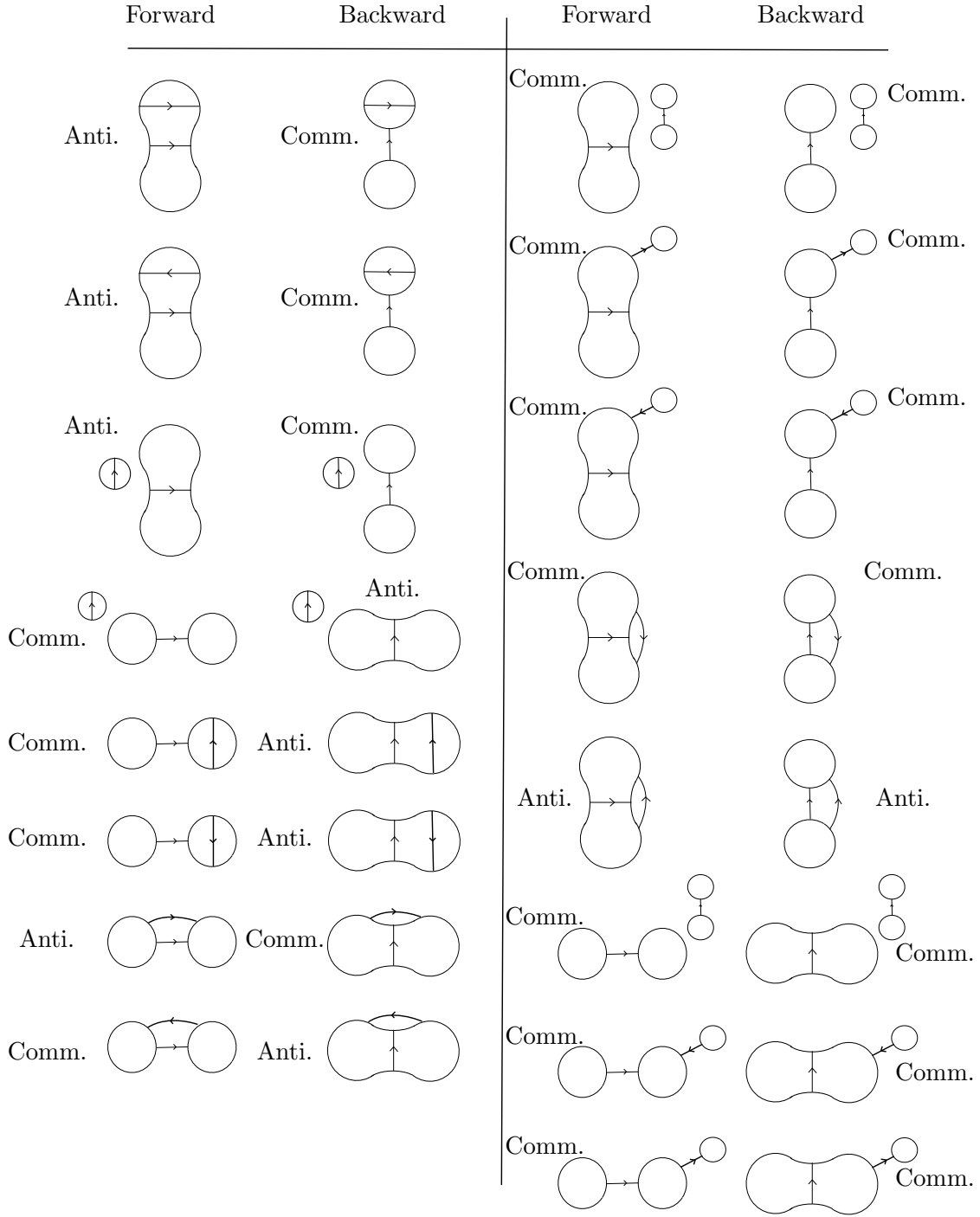
\begin{figure}
    \labellist
    \small \hair 2pt \pinlabel Anti. at 38 733 \pinlabel Anti. at 38
    631 \pinlabel Anti. at 38 550 \pinlabel Comm. at 10 416 \pinlabel
    Comm. at 10 350 \pinlabel Comm. at 10 290 \pinlabel Anti. at 10
    226 \pinlabel Comm. at 10 155 \pinlabel Comm. at 182 733 \pinlabel
    Comm. at 182 631 \pinlabel Comm. at 182 550 \pinlabel Anti. at 228
    447 \pinlabel Anti. at 160 350 \pinlabel Anti. at 160 290
    \pinlabel Comm. at 160 226 \pinlabel Anti. at 160 155 \pinlabel
    Comm. at 326 769 \pinlabel Comm. at 326 664 \pinlabel Comm. at 326
    562 \pinlabel Comm. at 325 458 \pinlabel Anti. at 324 315
    \pinlabel Comm. at 324 240 \pinlabel Comm. at 324 165 \pinlabel
    Comm. at 324 95 \pinlabel Comm. at 565 760 \pinlabel Comm. at 565
    668 \pinlabel Comm. at 565 568 \pinlabel Comm. at 550 458
    \pinlabel Anti. at 550 315 \pinlabel Comm. at 570 217 \pinlabel
    Comm. at 570 142 \pinlabel Comm. at 570 75 \pinlabel Forward at 88
    810 \pinlabel Backward at 228 810 \pinlabel Forward at 363 810
    \pinlabel Backward at 492 810
    \endlabellist
    \begin{center}
      \includegraphics[scale=0.7]{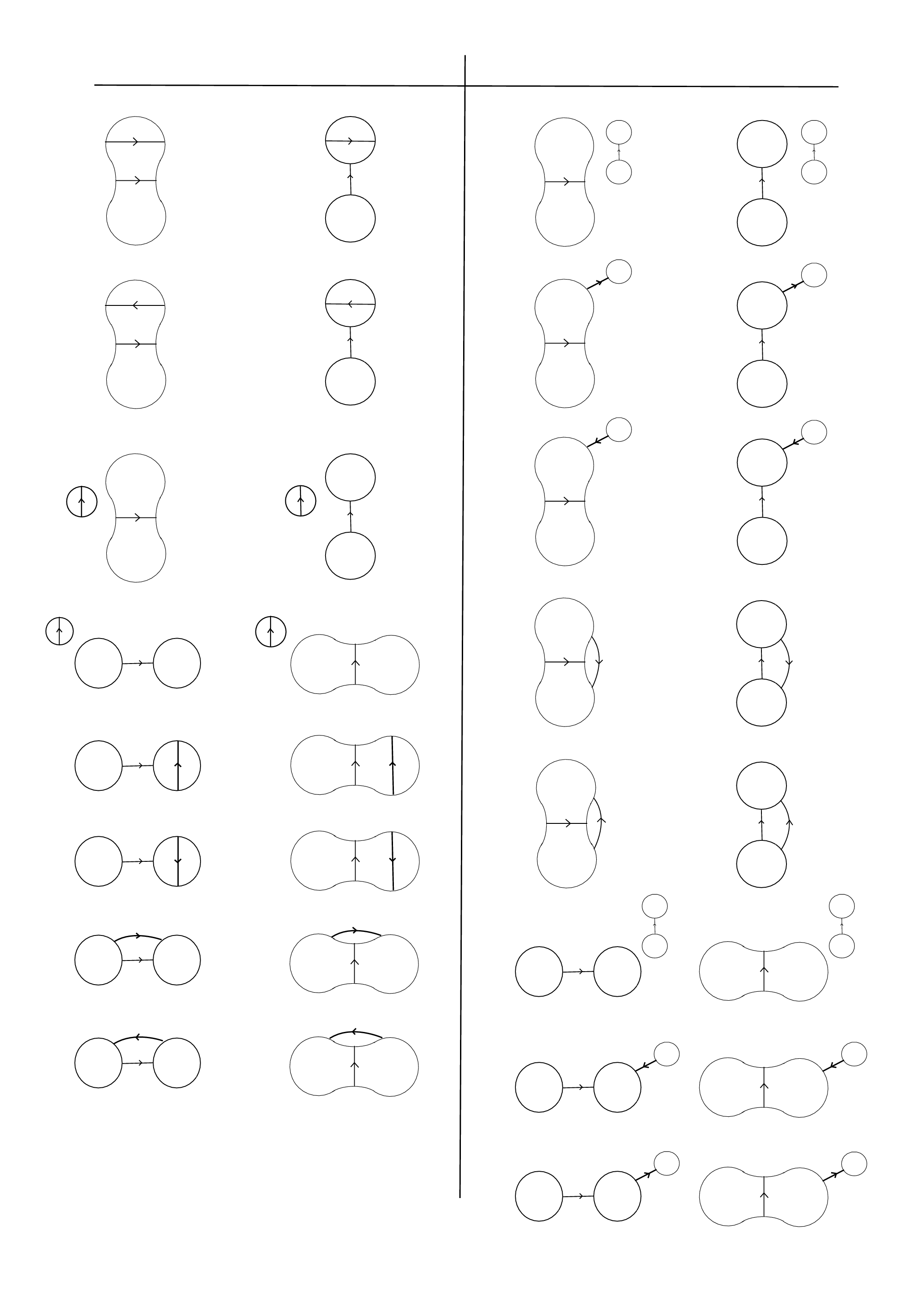}
    \end{center}
    \caption{The cases needed for Lemma~\ref{configs}.}
    \label{bubbles}
  \end{figure}

  Lemma~\ref{configs} immediately finishes the proof of
  Theorem~\ref{movemark}. Indeed, we know $\Delta\sigma(e) +
  \Delta\sigma(c) + a_f = 1$ modulo 2, since odd Khovanov homology
  satisfies $d^2 = 0$. Note that $\Delta\sigma(H) = \Delta\sigma(c)$
  when we are in the first case of Lemma~\ref{configs}, and
  $\Delta\sigma(H) = \Delta\sigma(c) + 1$ otherwise. Hence in either
  case we can conclude $\Delta\sigma(H) + \Delta\sigma(e) + a_b = 0$,
  as desired.
\end{proof}

\section{Spanning trees.}\label{spanningtrees}
As in \cite{Jaeger} and \cite{Roberts}, after inverting some of $R =
\Z[x_i]$, one can cancel the vertical differentials in the reduced
twisted complex to obtain a spanning-tree complex computing reduced
odd Khovanov homology. The situation will not be quite as nice as in
characteristic $2$, since the lack of a way to canonically determine
signs in odd Khovanov homology will lead to a sign ambiguity in the
spanning-tree differential. Still, we will discuss the situation
briefly.

Let $S$ denote the ring obtained from $R$ by inverting all products of
sums of the form $x_{i_1} + \ldots + x_{i_l}$ where $i_1, \ldots, i_l$
index some subset of the marks. Form the twisted complex with
coefficients in $S$. Consider a complete resolution $\rho$, with
circles $a_1, \ldots, a_k$. Let $w_i$ denote the sum of the variables
corresponding to marks on $a_i$. The complex $(V_{\rho}, d_{v,\rho})$
is actually the Koszul complex associated to the elements $w_1,
\ldots, w_k$ of $S$. We want to show this complex is acyclic.

In fact, the Koszul complex would already be acyclic over $R$, except
in the lowest degree (since the $w_i$ form a regular sequence in
$\Z[x_i]$). After tensoring with $S$, all the $w_i$ become invertible,
and the lowest homology of the complex (namely $S$ modulo the $w_i$)
is also trivial.

So after cancellation of $d_v$, we get a complex where the only
contributions come from connected resolutions, i.e. spanning trees of
the Tait graph of $D$. Consider two connected resolutions $T$ and
$\tilde{T}$, differing at only two crossings. There are two
intermediate two-component resolutions $\rho$ and $\rho'$ between
them. In each, one circle contains the basepoint. Let $w$ (resp. $w'$)
denote the sum of the variables of the marks on the circle without the
basepoint in $\rho$ (resp. $\rho'$). Then the contribution to
$\partial T$ in the spanning-tree complex, if we used the naive maps,
would be the sum of $T \rightarrow \rho \stackrel{1/w}{\rightarrow}
\rho \rightarrow \tilde{T}$ and $T \rightarrow \rho'
\stackrel{1/w'}{\rightarrow} \rho' \rightarrow \tilde{T}$. To put in
the actual signs, note that both $\rho$ and $\rho'$ come from
splitting circles in $T$. Hence the sign-corrections to $d_{v,\rho}$
and $d_{v,\rho'}$, and hence to $1/w$ and $1/w'$, are the same. The
other four maps, though, need to form an anticommuting square in the
cube of resolutions. The maps come from a 2-dimensional configuration
which is either $X$ or $Y$, depending on the orientations of the
relevant crossings. Recall that our convention was that $X$
anticommutes and $Y$ commutes. Hence if the configuration is $Y$, one
or three of the four maps must pick up a sign. Because of this sign,
the coefficient of the differential from $T$ to $\tilde{T}$ is $\pm
(1/w - 1/w')$. On the other hand, if the configuration is $X$, the
coefficient from $T$ to $\tilde{T}$ is $\pm (1/w + 1/w')$. The author
does not know a good way to decide between the $+$ and $-$ signs on
the outside without actually making explicit sign assignments in the
cube of resolutions.

\bibliographystyle{plain} \bibliography{biblio}

\begin{thebibliography}{1}

\bibitem{Baldwin}
John Baldwin.
\newblock On the spectral sequence from {K}hovanov homology to {H}eegaard
  {F}loer homology, 2009.
\newblock arXiv:0809.3293v4.

\bibitem{Jaeger}
Thomas Jaeger.
\newblock A remark on {R}oberts' totally twisted {K}hovanov homology, 2011.
\newblock arXiv:1109.1805.

\bibitem{Odd}
Peter Ozsv{\'a}th, Zolt{\'a}n Szab{\'o}, and Jacob Rasmussen.
\newblock Odd {K}hovanov homology, 2007.
\newblock arXiv:0710.4300.

\bibitem{Roberts}
Lawrence Roberts.
\newblock Totally twisted {K}hovanov homology, 2011.
\newblock arXiv:1109.0508.

\end{thebibliography}

\end{document}